\documentclass[a4paper,12pt]{amsart}
\usepackage{amsmath, amssymb, amsthm, xcolor, enumerate}
\usepackage{hyperref}

\newcounter{theorem}
\newtheorem{theorem}[theorem]{Theorem}
\newtheorem{lemma}[theorem]{Lemma}
\newtheorem{proposition}[theorem]{Proposition}

\theoremstyle{definition}
\newtheorem{definition}[theorem]{Definition}

\newtheorem{question}[theorem]{Question}

\theoremstyle{remark}
\newtheorem*{remark*}{Remark}
\newtheorem{remark}[theorem]{Remark}

\numberwithin{equation}{section}
\allowdisplaybreaks

\newcommand{\Z}{\mathcal Z}

\newcommand{\K}{\mathbb{K}}
\newcommand{\N}{\mathbb{N}}
\newcommand{\M}{\mathbb{M}}

\newcommand{\tr}{\mathrm{tr}}
\newcommand{\Tr}{\mathrm{Tr}}

\title{Stabilising uniform property $\Gamma$}
\author[J.\ Castillejos]{Jorge Castillejos}
\address{\hskip-\parindent Jorge Castillejos, Institute of Mathematics, Polish Academy of Sciences, ul. {\'S}niadeckich 8, 00-656 Warszawa, Poland}
\email{jcastillejoslopez@impan.pl}
\author[S.\ Evington]{Samuel Evington}
\address{\hskip-\parindent Samuel Evington, Mathematical Institute, University of Oxford, Oxford, OX2 6GG, UK.}
\curraddr{Mathematical Institute, University of M{\"u}nster, Einsteinstrasse 62, 48149 M{\"u}nster, Germany}
\email{evington@uni-muenster.de}

\thanks{SE was supported by EPSRC grant EP/R025061/2. JC was partially supported by long term structural funding -- a Methusalem grant of the Flemish Government}
\begin{document}

\maketitle

\begin{abstract}
We introduce stabilised property $\Gamma$, a $\mathrm{C}^*$-algebraic variant of property $\Gamma$ which is invariant under stable isomorphism. We then show that simple separable nuclear $\mathrm{C}^*$-algebras with stabilised property $\Gamma$ and $\mathrm{Cu}(A) \cong \mathrm{Cu}(A \otimes \Z)$ absorb the Jiang-Su algebra $\Z$ tensorially.
\end{abstract}

\renewcommand*{\thetheorem}{\Alph{theorem}}

\section*{Introduction}

A II$_1$ factor has property $\Gamma$ if there exist approximately central unitaries of trace
zero (\cite{vN43}). This notion was introduced by Murray and von Neumann as a tool to show the existence of non-hyperfinite $\mathrm{II}_1$ factors. (Indeed, the hyperfinite II$_1$ factor $\mathcal{R}$ has property $\Gamma$, whereas the free group factors $L(\mathbb{F}_n)$ do not.) Subsequently, Dixmier showed that property $\Gamma$ could also be characterised by the existence of non-trivial approximately central projections (\cite{Di69}).

\emph{Uniform property $\Gamma$} was recently introduced by the authors together with Tikuisis, White and Winter in order to obtain structural results for simple, separable, nuclear, unital $\mathrm{C}^*$-algebras (\cite{CETWW}). More precisely, it was shown that two regularity properties \emph{finite nuclear dimension} and \emph{$\Z$-stability} coincide as predicted by the Toms--Winter conjecture (\cite{TW07}). Further study of uniform property $\Gamma$, has shown that the full Toms--Winter conjecture holds for unital C$^*$-algebras with uniform property $\Gamma$ (\cite{CETW}). We refer to \cite{Win18} for a thorough description of the Toms--Winter conjecture and its importance in the classification programme of nuclear $\mathrm{C}^*$-algebras.

In this paper, we develop a variant of uniform property $\Gamma$ more suited to the analysis of non-unital C$^*$-algebras. We prove that it is invariant under stable isomorphism and coincides with uniform property $\Gamma$ for simple unital C$^*$-algebras. For clarity of presentation, we call this new property \emph{stabilised property $\Gamma$}. The key technical difference is that stabilised property $\Gamma$ involves all lower semicontinuous traces, not just the continuous ones. 

By Brown's theorem (\cite{Br77}), any non-zero hereditary $\mathrm{C}^*$-subalgebra of a simple separable $\mathrm{C}^*$-algebra will have stabilised property $\Gamma$ if the original $\mathrm{C}^*$-algebra has it. Using this, we obtain the following theorem that corresponds to the direction (iii) $\Rightarrow$ (ii) of the Toms--Winter conjecture in the non-unital setting. Previous results on this direction in the non-unital setting were obtained in \cite{Na13,Jac13}.

\begin{theorem}\label{thm:Main}
	Let $A$ be a simple, separable, nuclear $\mathrm{C}^*$-algebra such that $\mathrm{Cu}(A) \cong \mathrm{Cu}(A \otimes \Z)$. If $A$ has stabilised property $\Gamma$ then $A$ is $\Z$-stable.
\end{theorem}

The Cuntz semigroup hypothesis $\mathrm{Cu}(A) \cong \mathrm{Cu}(A \otimes \Z)$ is used to construct a hereditary $\mathrm{C}^*$-subalgebra of $A \otimes \K$ with favourable tracial properties and to ensure strict comparison of positive elements. (See Section \ref{subsection:Cuntz.equiv} for the relevant definitions and background material on the Cuntz semigroup.) In light of the recent developments in the theory of the Cuntz semigroup (\cite{Th20, APRT}), the condition $\mathrm{Cu}(A) \cong \mathrm{Cu}(A \otimes \Z)$ can be replaced with strict comparison in the stable rank one setting. 

Finally, by combining the work of many authors on the Toms--Winter conjecture (\cite{Ro04, Wi12, Ti14, MS12, MS14, SWW15, BBSTWW, CETWW, CE, CETW}), we obtain that the following form of the Toms--Winter conjecture, where strict comparison is replaced with $\mathrm{Cu}(A) \cong \mathrm{Cu}(A \otimes \Z)$, holds under the extra condition of stabilised property $\Gamma$.

\begin{theorem}\label{thm:Toms-Winter}
	Let $A$ be simple, separable, nuclear and non-elementary $\mathrm{C}^*$-algebra.
	The following statements are equivalent:
	\begin{enumerate}[(i)]
		\item $A$ has finite nuclear dimension,
		\item $A$ is $\Z$-stable,
		\item $A$ has stabilised property $\Gamma$ and $\mathrm{Cu}(A) \cong \mathrm{Cu}(A \otimes \Z)$.
	\end{enumerate}
\end{theorem}

\subsection*{Structure of the paper.} In Section 1, we review the preliminary material needed for this note. In Section 2, we define the notion of stabilised property $\Gamma$ and prove that it is invariant under stable isomorphism.
Finally, Theorems \ref{thm:Main} and \ref{thm:Toms-Winter} are proved in Section 3.  

\subsection*{Acknowledgements} Portions of this research were undertaken during the
workshop ``Topology and Measure in Dynamics and Operator Algebras'' at the Banff International Research Station. We thank the organisers and funders of this programme.
The authors would also like to thank Jamie Gabe, G\'abor Szab\'o, Hannes Thiel and Stuart White for useful comments on the topic of this paper. JC would like to thank Adam Skalski for his support during the preparation of this document.

\numberwithin{theorem}{section}

\section{Preliminaries}

\subsection{Notation}
We write $\K$ for the $\mathrm{C}^*$-algebra of compact operators on the separable Hilbert space $\ell^2$. The standard matrix units are denoted $e_{ij}$ for $i,j \in \N$. We identify the matrix algebra $\M_n$ with the subalgebra of $\K$ generated by $e_{ij}$ for $1 \leq i,j \leq n$. We write $1_n$ for the unit of $\M_n$, so $(1_n)_{n \in \N}$ is an approximate unit for $\K$.

\subsection{Traces}\label{subsection:traces} 
Let $A$ be a $\mathrm{C}^*$-algebra. An (extended) trace is a map $\tau:A_+ \rightarrow [0,\infty]$ such that $\tau(0) = 0$, $\tau(\lambda_1 a_1 + \lambda_2 a_2) = \lambda_1 \tau(a_1) + \lambda_2 \tau(a_2)$ for $\lambda_i > 0$ and $a_i \in A_+$, and $\tau(a^*a) = \tau(aa^*)$ for all $a \in A$. The set $\mathrm{Dom}(\tau) = \mathrm{span}\{a\in A_+:\tau(a) < \infty\}$ is an ideal of $A$, and $\tau$ induces a linear map $\mathrm{Dom}(\tau) \rightarrow \mathbb{C}$. If $\mathrm{Dom}(\tau)$ is dense in $A$, then $\tau$ is said to be \emph{densely defined}. We write $T^+(A)$ for the cone of densely-defined, lower semicontinuous traces on $A$. We endow it with the topology of pointwise convergence on the Pedersen ideal of $A$.

It is well-known that a densely-defined, \emph{continuous} trace $\tau$ on $A$ can be uniquely extended to a (bounded) positive linear functional on $A$ satisfying the trace identity $\tau(a_1a_2) = \tau(a_2a_1)$. We refer to such traces as \emph{bounded traces} and denote the cone of bounded traces by $T_b(A)$. We write $T(A)$ for the convex set of tracial states, i.e.\ the bounded traces of norm 1.

We write $\Tr$ for the canonical trace on $\K$. Every lower semicontinuous trace $\tau$ on $A$ extends uniquely to a lower semicontinuous trace $\tau \otimes \Tr$ on $A \otimes \K$, where $(\tau \otimes \Tr) (\sum_{ij} a_{ij} \otimes e_{ij}) = \sum_i\tau(a_{ii})$ for all $\sum_{ij} a_{ij} \otimes e_{ij} \in (A \otimes \K)_+$; see for example \cite[Remark 2.27 (viii)]{BK04}.  Moreover, by uniqueness, every lower semicontinuous trace $\rho$ on $A \otimes \K$ is of the form $\tau \otimes \Tr$ with $\tau(a) = \rho(a \otimes e_{11})$.

\subsection{Ultraproducts and limit traces}  
We fix a free ultrafilter $\omega$ on $\N$ that will be used throughout the paper for all ultraproduct constructions. 

Given a C$^*$-algebra $A$, we write $\ell^\infty(A)$ for the C$^*$-algebra of bounded sequences in $A$.
The ultrapower of a $\mathrm{C}^*$-algebra $A$ is then given by 
\begin{equation}
A_\omega := \ell^\infty(A) / \{(a_n) \mid \lim_{n \to \omega} \|a_n\| = 0 \}.
\end{equation}
We identify $A$ with the subalgebra of $A_\omega$ arising from constant sequences. We will adopt a standard abuse of notation and denote elements in $A_\omega$ by a choice of a representative sequence $(a_n)$.

A tracial state $\tau$ on $A_\omega$ is called a \emph{limit trace} if there is a sequence $(\tau_n)$ of tracial states on $A$ such that $\tau((a_n))=\lim_{n \to \omega} \tau_n(a_n)$ for all $(a_n) \in A_\omega$. The set of limit traces on $A_\omega$ will be denoted $T_\omega(A)$.

Suppose $T(A)$ is non-empty, the \emph{trace kernel ideal} is given by
\begin{equation}
	J_A = \{x \in A_\omega \mid \tau(x^*x ) = 0 \text{ for all } \tau \in T_\omega(A) \}.
\end{equation}
The \emph{uniform tracial ultrapower} of $A$ is defined as
\begin{equation}
A^\omega = A_\omega / J_A.
\end{equation}
When $A$ is separable, $A^\omega$ is unital if and only if $T(A)$ is compact by \cite[Proposition 1.11]{CETWW}. The notation $T_\omega(A)$ will also be used for tracial states on $A^\omega$ coming from limit traces.
There is a canonical map $\iota: A \to A^\omega$ given by taking constant sequences. 
This need not be an embedding in general, but it will be whenever $T(A)$ is separating. 
Abusing notation slightly, we will simply write $A^\omega \cap A'$ instead of $A^\omega \cap \iota(A)'$. 

The following result was established by Kirchberg and R{\o}rdam in \cite{KR14}, building on an observation of Sato \cite{Sa11}.  The version stated here is a combination of \cite[Proposition 4.5(iii) and Proposition 4.6]{KR14}.\footnote{Note that the proof of \cite[Proposition 4.6]{KR14} also works in the non-unital case, provided one interprets $1$ as the unit of the minimal unitisation.}

\begin{proposition}[Central Surjectivity]\label{CentSurject}
Let $A$ be a separable $\mathrm{C}^*$-algebra with $T(A)$ compact and non-empty. Then the canonical map $A_\omega\cap A'\rightarrow A^\omega\cap A'$ is a surjection.
\end{proposition}

\subsection{Uniform Property Gamma}

We now recall the definition of uniform propery $\Gamma$ from \cite[Definition 2.1]{CETWW}.

\begin{definition}\label{defn:oldGamma}
Let $A$ be a separable $\mathrm{C}^*$-algebra with $T(A)$ non-empty and compact.
Then $A$ is said to have \emph{uniform property $\Gamma$} if for all $n\in\mathbb N$, there exist projections $p_1,\dots,p_n\in A^\omega \cap A'$ summing to $1_{A^\omega}$, such that
\begin{equation}\label{oldgamma.def}
\tau(ap_i)=\frac{1}{n}\tau(a),\quad a\in A,\ \tau\in T_\omega(A),\ i=1,\dots,n.
\end{equation}
\end{definition} 
Examples of $\mathrm{C}^*$-algebras with uniform property $\Gamma$ include $\Z$-stable $\mathrm{C}^*$-algebras and the non-$\Z$-stable Villadsen algebras of the first type. This definition is explored in detail in \cite{CETW}.

\subsection{Lifting Results} \label{prelim.lifting}
It was proved in \cite[Proposition 2.6]{AP77} that a finite set of orthogonal positive contractions in a quotient can always be lifted to a finite set of orthogonal positive contractions. 
From this, Loring deduced that the universal $\mathrm{C}^*$-algebra generated by $n$ orthogonal positive contractions $C_0(0,1]\otimes \mathbb{C}^{n}$ is projective \cite[Theorem 4.6]{Lo93}. 
In fact, for any finite-dimensional $\mathrm{C}^*$-algebra $F$, the $\mathrm{C}^*$-algebra $C_0(0,1] \otimes F$ is projective \cite[Theorem 4.9]{Lo93}. In other words, every $^*$-homomorphism $C_0(0,1] \otimes F \rightarrow A/I$ can be lifted to $^*$-homomorphism $C_0(0,1] \otimes F \rightarrow A$.

\subsection{Kirchberg's central sequence algebra}

Let $A$ be a $\sigma$-unital $\mathrm{C}^*$-algebra. Kirchberg's central sequence algebra of $A$, as introduced in \cite{Ki06}, is given by
\begin{align}
F_\omega (A) := (A_\omega \cap A')/ \mathrm{Ann}(A),
\end{align} 
where $\mathrm{Ann}(A) :=\{x \in A_\omega \mid xa = ax = 0 \text{ for all }a \in A\}$. As with $A_\omega$, we will abuse notation and use sequences $(a_n)$ to denote elements of $F_\omega (A)$.
Observe that any approximate unit $(e_n)$ for $A$ represents the unit of $F_\omega (A)$. 

Kirchberg proved that $F_\omega(A)$ is a stable invariant, i.e. $F_\omega(A)  \cong F_\omega(A \otimes \K)$ (see \cite[Appendix A.1]{Ki06}). For our purposes it will be important to have an explicit description of this isomorphism. 

\begin{lemma}\label{prop:central.seq.algebra.stable}
	Let $A$ be a $\sigma$-unital $\mathrm{C}^*$-algebra.
	The $^*$-homomorphism $\Phi:A_\omega \rightarrow (A \otimes \K)_\omega$ given by $(x_n) \mapsto (x_n \otimes 1_n)$ maps $A_\omega \cap A'$ into $(A \otimes \K)_\omega \cap (A \otimes \K)'$ and induces a $^*$-isomorphism $\bar{\Phi}:F_\omega(A) \cong F_\omega(A \otimes \K)$.
\end{lemma}
\begin{proof}
The only assertion that is not immediate is the surjectivity of $\bar{\Phi}$. Let $z = (z_n) \in (A \otimes \K)_\omega \cap (A \otimes \K)'$. Write $z_n = \sum_{i,j} z^{(n)}_{ij} \otimes e_{ij}$. Set $x = (z^{(n)}_{11})$. Since $z$ commutes with $a \otimes e_{11}$ for all $a \in A$, it follows that $x \in A_\omega \cap A'$ by considering the $(1,1)$ entry of $z_n(a \otimes e_{11}) - (a \otimes e_{11})z_n$. 
	
We need to show that $z$ and $\Phi(x) = (z^{(n)}_{11} \otimes 1_n)$ agree mod $\mathrm{Ann}(A \otimes \K)$. Let $a \in A$ and $r,s \in \N$. Since $z, \Phi(x) \in (A \otimes \K)'$, we have
\begin{align}
	(a^4 \otimes e_{rs}) \, z &= (a \otimes e_{r1}) \, (a \otimes e_{11}) \, z \,(a \otimes e_{11}) \, (a \otimes e_{1s}) \notag\\
		&= (a \otimes e_{r1}) \, (a \otimes e_{11}) \, \Phi(x) \,(a \otimes e_{11}) \, (a \otimes e_{1s}) \notag\\
		&= (a^4 \otimes e_{rs}) \, \Phi(x). 
\end{align}	  
The result follows as $\mathrm{span}\{a^4 \otimes e_{rs}: a \in A,r,s \in \N\}$ is dense in $A \otimes \K$.		
\end{proof}

An alternative description of the isomorphism $\bar{\Phi}$ using multiplier algebras is given in \cite[Proposition 1.9]{BS16}.

\subsection{Cuntz equivalence and strict comparison}\label{subsection:Cuntz.equiv}

Let $a,b \in A_+$. It is said that $a$ is \emph{Cuntz sub-equivalent} to $b$, denoted $a \precsim b$, if there exists a sequence $(x_n)$ such that $a = \lim_{n \to \infty} x_n^* b x_n$.
A positive element $a$ is \emph{Cuntz equivalent} to $b$ if $a \precsim b$ and $b \precsim a$. We denote this by $a \sim b$ and write $[a]$ for the Cuntz equivalence class of a positive element $a$.
The \emph{Cuntz semigroup} of $A$ is defined as $\mathrm{Cu}(A):= (A \otimes \K)_+ / \sim$ equipped with orthogonal addition and the order induced by Cuntz sub-equivalence. We refer to \cite{To11} for more details about this construction.

The cone of functionals on $\mathrm{Cu}(A)$ is denoted by $F(\mathrm{Cu}(A))$ and its dual cone by $L(F(\mathrm{Cu}(A)))$; we refer to \cite[Section 4]{ERS11} for the corresponding definitions. 
The \emph{rank function} $d_\tau: (A \otimes \K)_+ \to \mathbb{R}$ associated to a lower semicontinuous trace $\tau$ is given by $d_\tau(a) = \lim_{n \to \infty} \tau(a^{1/n})$.
Every rank function $d_\tau$ induces a functional on $\mathrm{Cu}(A)$, and by \cite[Proposition 4.2]{ERS11}, the map $\tau \mapsto d_\tau$ is a bijection between the cone of lower semicontinuous quasitraces on $A$ and $F(\mathrm{Cu}(A))$. By a celebrated theorem of Haagerup, all lower semicontinuous quasitraces on an exact C$^*$-algebra are traces.\footnote{See \cite{Ha14} for the unital case and \cite[Remark 2.29(i)]{BK04} for how to deduce the general case from the unital case.} 

The Cuntz semigroup $\mathrm{Cu}(A)$ is said to be \emph{almost unperforated} if $[a] \leq [b]$ whenever $(k+1)[a] \leq k[b]$ for some $k \in \mathbb{N}$. Suppose $A$ is simple and exact with at least one trace. Then $\mathrm{Cu}(A)$ is almost unperforated if and only if $A$ has \emph{strict comparison}, i.e.\ $[a] \leq [b]$ whenever $d_\tau([a]) < d_\tau([b])$ for all $\tau \in T^+(A)$ that satisfy $d_\tau([b])=1$; see \cite[Proposition 2.1]{OPR} for example.

\section{Stabilising uniform property gamma}

\subsection{Generalised traces on ultraproducts}

In this section, we recall the definition of generalised limit traces and how they can be used to define bounded traces on the Kirchberg central sequence algebra.

\begin{definition}[{\cite[Definition 2.1]{Szabo19}}]
	A trace $\tau: \ell^\infty(A)_+ \to [0, \infty]$ is called a \emph{generalised limit trace} if there is a sequence $(\tau_n)$ of lower semicontinuous traces on $A$ such that
	\begin{equation}\label{eqn:gen_limit_trace}
	\tau((a_n)) = \sup_{\epsilon>0} \lim_{n \to \omega} \tau_n((a_n-\epsilon)_+), \qquad (a_n)\in \ell^\infty(A)_+.
	\end{equation}
	A trace on $A_\omega$ is  called a generalised limit trace if it is induced by a generalised limit trace on $\ell^\infty(A)$. Write $\widetilde{T}_\omega(A)$ to denote the set of generalised limit traces on $A_\omega$.
\end{definition}
\begin{remark}
	The left hand side of \eqref{eqn:gen_limit_trace} is the \emph{lower semicontinuous regularisation} of the trace given by $\tau((a_n)) = \lim_{n\to\omega} \tau(a_n)$; see \cite[Remark 2.27 (iv)]{BK04}. This step is necessary to ensure that generalised limit traces are lower semicontinuous and that the induced trace on $A_\omega$ is well defined. 
\end{remark}

There is no need to consider generalised limit traces in the simple unital setting, as all non-trivial examples of generalised limit traces can be obtained as multiples of limit traces. We record this below in a slightly more general form. At the other extreme, for stable $\mathrm{C}^*$-algebras there are no limit traces, but generalised limit traces may exist.

\begin{proposition}\label{prop:LimitTracesSuffice}
	Let $A$ be a simple separable $\mathrm{C}^*$-algebra with $T^+(A) \neq \{0\}$. Suppose $T^+(A) = T_b(A)$ and $T(A)$ is compact (as happens when $A$ is unital). Then every generalised limit trace $\tau \in \widetilde{T}_\omega(A)$ that is finite on some non-zero positive element of $A$ is a multiple of a limit trace. 
\end{proposition}
\begin{proof}
Let $(\tau_n)$ be a sequence of lower semicontinuous traces on $A$ that induce $\tau$. Write $\tr_\infty$ for the lower semicontinuous trace on $A$ with $\tr_\infty(a) = \infty$ for all non-zero $a \in A_+$. If it were true that $\{n \in \N: \tau_n = \tr_\infty\} \in \omega$, then $\tau(a) = \infty$ for all non-zero $a \in (A_\omega)_+$, contradicting our hypothesis. Hence, $\{n \in \N: \tau_n \neq \tr_\infty\} \in \omega$, since $\omega$ is a ultrafilter. Therefore, we may assume without loss of generality that $\tau_n \neq \tr_\infty$ for all $n \in \N$. Since $A$ is simple and $T^+(A) = T_b(A)$, it follows that all the $\tau_n$ are bounded.   

Suppose now that $\tau(a)>0$ for some non-zero $a \in A_+$.
Since $T(A)$ is compact and non-empty, $\alpha = \min_{\sigma \in T(A)} \sigma(a)$ exists and simplicity ensures that $\alpha > 0$. For all $n \in \N$, we have $\tau_n(a) \geq \alpha\|\tau_n\|$. So $\lim_{n \to \omega} \|\tau_n\| \leq \alpha^{-1}\tau(a) < \infty$. Hence, $\tau$ is bounded with $\|\tau\| \leq \lim_{n \to \omega} \|\tau_n\|$. In fact, this is an equality, as there is a sequence $(a_n)$ of contractions with $|\tau_n(a_n)| > \|\tau_n\| - \tfrac{1}{n}$. Let $\bar{\tau}_n \in T(A)$ be the tracial state obtained by normalising $\tau_n$. (If $\tau_n = 0$ then $\bar{\tau}_n$ can be chosen arbitrarily.) Let $\bar{\tau}$ be the limit trace corresponding to $(\bar{\tau}_n)$. Then $\tau = \|\tau\|\bar{\tau}$.
\end{proof}

As indicated in \cite[Remark 2.3]{Szabo19}, generalised limit traces can be used to define bounded traces on $F_\omega(A)$.
\begin{definition}
For every $\tau \in \widetilde{T}_\omega(A)$ and $a \in A_+$ define an (extended) trace by
\begin{equation}
\tau_a : (A_\omega \cap A')_+ \to [0, \infty], \quad x \mapsto \tau(ax).
\end{equation}
When $\tau(a)< \infty$, we have $\tau_a(x) \leq \|x\|\tau(a)$, so the trace $\tau_a$ is bounded. Moreover, $\tau_a(\mathrm{Ann}(A))=0$. Thus, there is also an induced a trace on $F_\omega (A)$ with the same norm. This trace will also be denoted by $\tau_a$.
\end{definition}

Observe that if $A$ is $\sigma$-unital with approximate unit $(e_n)$, we obtain $\|\tau_a\| = \tau(a)$ by taking $x = (e_n)$.

\subsection{Stabilised property gamma}

In this section, we introduce our definition of stabilised property $\Gamma$. We shall then establish its relationship to uniform property $\Gamma$. Finally, we shall prove that stabilised property $\Gamma$ is invariant under stable isomorphism.

\begin{definition}\label{def:gamma}
	Let $A$ be a separable $\mathrm{C}^*$-algebra. Then $A$ has \textit{stabilised property $\Gamma$} if for any $n \in \N$ there exist pairwise orthogonal positive contractions $e_1, \ldots, e_n \in F_\omega(A)$ such that for all $a \in A_+$ and $\tau \in \widetilde{T}_\omega(A)$ with $\tau(a) < \infty$,
	\begin{equation}
	 \tau_a(e_i) = \frac{1}{n} \tau(a),\quad \quad i=1,\ldots,n.
	\label{def.prop.Gamma.eq}
	\end{equation}
\end{definition}

We now record the relationship between stabilised property $\Gamma$ and uniform property $\Gamma$ in the simple setting.  

\begin{proposition}\label{prop:sGamma.implies.uGamma}
Let $A$ be a simple separable $\mathrm{C}^*$-algebra with $T(A)$ compact and non-empty.
\begin{itemize}
	\item[(i)] If $A$ has stabilised property $\Gamma$ then $A$ has uniform property $\Gamma$.
	\item[(ii)] The reverse implication holds whenever $A$ is unital or more generally whenever $T^+(A) = T_b(A)$.
\end{itemize}
\end{proposition}
\begin{proof}
	(i): Suppose $A$ has stabilised property $\Gamma$. Given $n \in \N$, let $e_1,\ldots,e_n \in F_\omega(A)$ be orthogonal positive contractions such that \eqref{def.prop.Gamma.eq} holds. Lifting the $e_i$ to orthogonal positive contractions in $A_\omega \cap A'$ and then descending to $A^\omega \cap A'$, we get orthogonal positive contractions $e_1',\ldots,e_n' \in A^\omega \cap A'$. Since limit traces are a subset of the generalised limit traces, we have 
\begin{equation}
	\tau(ae_i') = \tau_a(e_i) = \frac{1}{n}\tau(a)
\end{equation}  
for all $a \in A_+$ and $\tau  \in T_\omega(A)$. By \cite[Lemma 1.10]{CETWW}, we have $\sum_i e_i' = 1_{A^\omega}$. By \cite[Proposition 1.7]{CETW}, the $e_i'$ are projections. Therefore, $A$ has uniform property $\Gamma$.

(ii): Suppose $A$ has uniform property $\Gamma$. Let $p_1,\ldots,p_n \in A^\omega \cap A'$ be orthogonal projections summing to $1_{A^\omega}$ as in Definition \ref{defn:oldGamma}. First, lift the $p_i$ to orthogonal positive contractions in $A_\omega \cap A'$ using Proposition \ref{CentSurject} and Section \ref{prelim.lifting}. Then descend to $F_\omega(A)$ to get orthogonal positive contractions $e_1,\ldots,e_n \in F_\omega(A)$.

Whenever $A$ is unital, or more generally whenever $T^+(A) = T_b(A)$, Proposition \ref{prop:LimitTracesSuffice} ensures that all non-trivial generalised limit traces are multiples of limit traces. Thus, it suffices to prove that for all $a \in A$ and $\tau \in T_\omega(A)$
\begin{equation}
 	\tau_a(e_i) = \frac{1}{n} \tau(a),\quad \quad i=1,\ldots,n.
\end{equation}
But this is equivalent to \eqref{oldgamma.def}.
\end{proof}

An example of a C$^*$-algebra with uniform property $\Gamma$ that does not have stabilised property $\Gamma$ is $(\K \otimes C^*_r(\mathbb{F}_2)) \oplus \M_{2^\infty}$, where $C^*_r(\mathbb{F}_2)$ is the reduced group C$^*$-algebra of the free group $\mathbb{F}_2$ and $\M_{2^\infty}$ is the CAR algebra. Indeed, the unbounded trace on $\K \otimes C^*_r(\mathbb{F}_2)$ will induce a generalised limit trace for which \eqref{def.prop.Gamma.eq} fails. 
(This follows from the fact that the group von Neumann algebra $\mathrm L(\mathbb{F}_2)$ is a II$_1$ factor that does not have property $\Gamma$ \cite{vN43}.) It would be interesting to know if this phenomenon can occur in the simple setting. 

\begin{question}
	Is there a simple $\mathrm{C}^*$-algebra with uniform property $\Gamma$ that does not have stabilised property $\Gamma$?
\end{question}

We now turn to proving that stabilised property $\Gamma$ is a stable invariant. On first glance this may appear to be immediate on account of the fact that $F_\omega(A) \cong F_\omega(A \otimes \K)$. However, we must also understand what such an isomorphism does to traces of the form $\tau_a$ for $\tau \in \widetilde{T}_\omega(A)$ and $a \in A_+$ with $\tau(a) < \infty$. 

Ultimately, we will need to carefully analyse the isomorphism $\bar{\Phi}$ of Lemma \ref{prop:central.seq.algebra.stable}. As a warm up, we first consider isomorphisms $\Psi:A \rightarrow B$, where the result follows from the functoriality of the various constructions.

\begin{proposition}\label{prop.isomorphism}
	Let $\Psi:A \rightarrow B$ be an isomorphism of $\mathrm{C}^*$-algebras, and write $\bar{\Psi}:F_\omega(A) \rightarrow F_\omega(B)$ for the induced isomorphism. Let $\sigma \in \widetilde{T}_\omega(B_\omega)$ be induced by the sequence of lower semicontinuous traces $(\sigma_n)$. Set $\tau_n = \sigma_n \circ \Psi$ for each $n \in \N$ and let $\tau \in \widetilde{T}_\omega(A_\omega)$ be the induced generalised limit trace. Then 
	\begin{equation}
	\sigma_{\Psi(a)} \circ \bar{\Psi} = \tau_a
	\end{equation}
	for all $a \in A$ with $\tau(a) < \infty$. 
\end{proposition}  
\begin{proof}
	Let $x \in F_\omega(A)_+$ be represented by the sequence of positive elements $(x_n)$. Unwinding the definitions, we get
	\begin{align}
		\tau_a(x) &= \sup_{\epsilon > 0} \lim_{n\to\omega} \tau_n((a^{1/2}x_na^{1/2} - \epsilon)_+) \notag \\
		&= \sup_{\epsilon > 0} \lim_{n\to\omega} \sigma_n(\Psi((a^{1/2}x_na^{1/2} - \epsilon)_+)) \notag \\
		&= \sup_{\epsilon > 0} \lim_{n\to\omega}  \sigma_n((\Psi(a^{1/2})\Psi(x_n)\Psi(a^{1/2})-\epsilon)_+)  \\
		&= \sigma_{\Psi(a)}(\bar{\Psi}(x)). \qedhere
	\end{align}
\end{proof}

We now proceed with the main technical result. We write $\bar{\Phi}:F_\omega(A) \rightarrow F_\omega(A \otimes \K)$ for the $^*$-isomorphism from Lemma \ref{prop:central.seq.algebra.stable} and, to facilitate computations, we shall also use the map $\Phi:A_\omega \rightarrow (A \otimes \K)_\omega$ that induces $\bar{\Phi}$.

\begin{proposition}\label{prop.stablisation}
	Let $A$ be a separable $\mathrm{C}^*$-algebra and let $(\sigma_n)$ be a sequence of lower semicontinuous traces on $A \otimes \K$. Write $\sigma_n = \tau_n \otimes \Tr$ with $(\tau_n)$ lower semicontinuous traces  on $A$.
	Let $\sigma$ and $\tau$ be the generalised limit traces induced by $(\sigma_n)$ and $(\tau_n)$, respectively. 
	Let $a \in (A \otimes \K)_+$ and write $a = \sum_{ij} a_{ij} \otimes e_{ij}$.
	
Suppose $\sigma(a) < \infty$. Then $\tau(a_{ii}) < \infty$ for all $i \in \mathbb{N}$. Moreover,
\begin{equation}
	\sigma_a \circ \bar{\Phi} = \sum_{i=1}^\infty \tau_{a_{ii}},
\end{equation}
where the sum converges in the norm topology on $T_b(F_\omega(A))$.
\end{proposition}

\begin{proof}
We view $A_\omega \otimes \K$ as a subalgebra of $(A \otimes \K)_\omega$. Since $\sigma$ restricts to a lower semicontinuous trace on $A_\omega \otimes \K$, we have 
\begin{equation}\label{eqn:trOnStab}
	\sigma\left(\sum_{ij}z_{ij} \otimes e_{ij}\right) = \sum_i \sigma(z_{ii} \otimes e_{11})
\end{equation}
for any positive element $\sum_{ij}z_{ij} \otimes e_{ij} \in (A_\omega \otimes \K)_+$; see Section \ref{subsection:traces}.

We first consider the case that $a \in (A \otimes \M_N)_+$ for some $N \in \N$, so $a = 
\sum_{i,j=1}^N a_{ij} \otimes e_{ij}$. Let $\bar{x} \in F_\omega(A)_+$. Fix a positive lift $x \in A_\omega \cap A'$ of $\bar{x}$ and a representative sequence of positive elements $(x_n)$ for $x$. Recall that $\Phi$ is given by $(x_n) \mapsto (x_n \otimes 1_n )$. Then 
\begin{align}
	\sigma_a(\bar{\Phi}(\bar{x})) &= \sigma(a\Phi(x)) \notag\\
	&= \sigma\left(\left(\sum_{i,j=1}^N a_{ij}x_n \otimes e_{ij}\right)\right) \notag \\
	&= \sigma\left(\sum_{i,j=1}^N (a_{ij}x_n) \otimes e_{ij}\right) \notag \\		
	&\hspace{-1.5mm} \stackrel{\eqref{eqn:trOnStab}}{=} \sum_{i=1}^N \sigma((a_{ii}x_n) \otimes e_{11}), \notag \\
	&= \sum_{i=1}^N \sigma((a_{ii}x_n \otimes e_{11}))  
\end{align}
where in the second line we have used that $1_n$, the unit of $\M_n$, acts as unit on $e_{ij}$ whenever $n \geq N \geq i,j$, and that the first $N$ terms of the sequence do not affect which element of the ultrapower is represented. 
 
Since $e_{11}$ is a projection, we have
\begin{align}
	\sigma((a_{ii}x_n \otimes e_{11}))
	&=\sigma((a_{ii}^{1/2}x_na_{ii}^{1/2} \otimes e_{11})) \notag \\ 
	&= \sup_{\epsilon>0} \lim_{n\rightarrow\omega} \sigma_n ((a_{ii}^{1/2}x_na_{ii}^{1/2} \otimes e_{11} - \epsilon)_+) \notag \\
	&=  \sup_{\epsilon>0} \lim_{n\rightarrow\omega} \sigma_n ((a_{ii}^{1/2}x_na_{ii}^{1/2} - \epsilon)_+ \otimes e_{11}) \notag \\
	&=  \sup_{\epsilon>0} \lim_{n\rightarrow\omega} \tau_n ((a_{ii}^{1/2}x_na_{ii}^{1/2} - \epsilon)_+) \notag \\
	&= \tau((a_{ii}^{1/2}x_na_{ii}^{1/2}))\notag \\
	&= \tau((a_{ii}x_n)).
\end{align}
Hence, $\sigma_a(\bar{\Phi}(\bar{x})) = \sum_{i=1}^N \tau_{a_{ii}}(\bar{x})$.

We now consider the general case where $a \in (A \otimes \K)_+$. Write $a = \sum_{i,j} a_{ij} \otimes e_{ij}$ and let $p_N = 1_{A^\sim} \otimes 1_N \in A^\sim \otimes \K$.  As $\sigma$ is a lower semicontinuous trace and $\Phi$ maps central sequences to central sequences, we have
\begin{align}
	\sigma_a(\bar{\Phi}(\bar{x})) &= \sigma(a\Phi(x)) \notag \\
	&=  \lim_N \sigma(a^{1/2}p_N a^{1/2}\Phi(x)) \notag \\
	&=  \lim_N \sigma(a^{1/2}p_N\Phi(x)^{1/2}\Phi(x)^{1/2} p_N a^{1/2}) \notag \\
	&= \lim_N \sigma(\Phi(x)^{1/2}p_Na^{1/2}a^{1/2}p_N\Phi(x)^{1/2}) \notag \\
	&= \lim_N \sigma(p_Nap_N\Phi(x)) \notag \\
	&= \lim_N \sigma_{p_Nap_N}(\bar{\Phi}(\bar{x})).
\end{align}
 
Combining this with the first part of the proof, we obtain that $\sigma_a(\bar{\Phi}(\bar{x})) = \sum_{i=1}^\infty \tau_{a_{ii}}(\bar{x})$. Finally, since $\sum_{i=1}^\infty \|\tau_{a_{ii}}\| = \sum_{i=1}^\infty \tau(a_{ii}) = \sigma(a) < \infty$, the series $\sum_{i=1}^\infty \tau_{a_{ii}}$ is (absolutely) norm convergent in $F_\omega(A)^*$.
\end{proof}

\begin{theorem}\label{thm:upGamma.stable}
	Stabilised property $\Gamma$ is preserved under stable isomorphism, i.e.\ if $A$ has stabilised property $\Gamma$ and $A \otimes \K \cong B \otimes \K$, then $B$ has stabilised property $\Gamma$.
\end{theorem}
\begin{proof} 
By Proposition \ref{prop.isomorphism}, stabilised property $\Gamma$ is invariant under isomorphism. Hence, it suffices to show that $A$ has stabilised property $\Gamma$ if and only if $A \otimes \K$ does too. 

Suppose $e_1,\ldots,e_n \in F_\omega(A)$ are pairwise orthogonal positive contractions satisfying \eqref{def.prop.Gamma.eq} for all $a \in A_+$ and $\tau \in \widetilde{T}_\omega(A)$ with $\tau(a) < \infty$. Set $f_i = \bar{\Phi}(e_i) \in F_\omega(A \otimes \K)$. Then $f_1, \ldots, f_n$ are pairwise orthogonal positive contractions. Moreover, by Proposition \ref{prop.stablisation}, we have
\begin{equation} \label{eq.stab}
	\sigma_a(f_i) = \frac{1}{n}\sigma(a) 
\end{equation} 
for all $a \in (A \otimes \K)_+$ and $\sigma \in \widetilde{T}_\omega(A \otimes \K)$ with $\sigma(a) < \infty$. Therefore stabilised property $\Gamma$ passes to the stabilisation.

Conversely, suppose $f_1, \ldots, f_n$ are pairwise orthogonal positive contractions satisfying \eqref{eq.stab} for all $a \in (A \otimes \K)_+$ and $\sigma \in \widetilde{T}_\omega(A \otimes \K)$ with $\sigma(a) < \infty$. Set $e_i = \bar{\Phi}^{-1}(f_i)$. 
Then $e_1, \ldots, e_n$ are pairwise orthogonal positive contractions. Moreover, by Proposition \ref{prop.stablisation}, it follows that \eqref{def.prop.Gamma.eq} holds for all $\tau \in \widetilde{T}_\omega(A)$ for all $a \in A_+$ and $\tau \in \widetilde{T}_\omega(A)$ with $\tau(a) < \infty$. Indeed, if $\tau$ is induced by $(\tau_n)$, then take $\sigma$ to be induced by $(\tau_n \otimes \Tr)$ and note that $\sigma_{a \otimes e_{11}} \circ \bar{\Phi} = \tau_a$ by Proposition \ref{prop.stablisation}. Therefore stabilised property $\Gamma$ passes from $A \otimes \K$ to $A$.
\end{proof}

\begin{remark}
	It follows from Proposition \ref{prop.stablisation} and the proof of Theorem \ref{thm:upGamma.stable} that $\bar{\Phi}$ induces an affine homeomorphism between the weak$^*$-closed convex hull of traces in $F(A)$ of the form $\tau_a$ (with $\tau \in \widetilde{T}_\omega(A), a \in A_+$ such that $\tau(a)< \infty$)
	and the weak$^*$-closed convex hull of traces in $F(A \otimes \mathbb{K})$ of the form $\sigma_b$ (with $\sigma \in \widetilde{T}_\omega(A\otimes \mathbb{K}), b \in (A \otimes \mathbb{K})_+$ such that $\sigma(b)< \infty$). This is underlying principle behind the proof of invariance under stable isomorphism. 
\end{remark}

\section{$\Z$-stability}

In this section, we shall consider stabilised property $\Gamma$ in the setting of the Toms--Winter conjecture, and we shall prove Theorems \ref{thm:Main} and \ref{thm:Toms-Winter}. We begin by recalling the definition of the uniform McDuff property.

\begin{definition}[{\cite[Definition 4.2]{CETW}}]
	Let $A$ be a separable $\mathrm{C}^*$-algebra with $T(A)$ non-empty and compact. We say that $A$ is \emph{uniformly McDuff} if for each $n \in \mathbb{N}$ there exists a unital embedding $\M_n \to A^\omega \cap A'$.
\end{definition}
 
In the proof of Theorem \ref{thm:Main}, we shall access $\Z$-stability using \cite[Theorem 2.1]{Jac13}, which is based on \cite{MS12, Na13, Ki06, RW10}. We restate this result below for the benefit of the reader using the notation and terminology of this paper.{\footnote{Combining central surjectivity (Proposition \ref{CentSurject}) with the projectivity of $C(0,1] \otimes \M_n$ (Section \ref{prelim.lifting}), the uniform McDuff property is equivalent to the existence of \emph{uniformly tracially large order zero maps} $\M_n \rightarrow A_\omega \cap A'$; see \cite[Remark 4.3]{CETW}.}

\begin{theorem}[{cf. \cite[Theorem 2.1]{Jac13}}]\label{thm:SI}
	Let $A$ be a separable, simple, non-elementary, nuclear $\mathrm{C}^*$-algebra with strict comparison such that $T^+(A) = T_b(A)$ and $T(A)$ is non-empty and compact. If $A$ is uniformly McDuff then it is $\Z$-stable.
\end{theorem}

For the proof of Theorem \ref{thm:Main}, we shall also need the following lemma, which states that, under suitable conditions,  a simple stably projectionless $\mathrm{C}^*$-algebra is stably isomorphic to a $\mathrm{C}^*$-algebra with a well-behaved tracial state space.

\begin{lemma}\label{lem:nice.her.subalgebra}
	Let $A$ be a  simple, separable, stably projectionless, exact $\mathrm{C}^*$-algebra such that $\mathrm{Cu}(A) \cong \mathrm{Cu}(A \otimes \Z)$. 
	Then $A$ is stably isomorphic to a $\mathrm{C}^*$-algebra $B$ with $T^+(B) = T_b(B)$ and with $T(B)$ non-empty and compact.
\end{lemma}
\begin{proof}	
	As $A$ is stably projectionless, there exists a quasitrace on $A$ by \cite[Theorem 1.2]{BC82}. As $A$ is exact, Haagerup's Theorem ensures all lower semicontinuous quasitraces on $A$ are traces. Since $A$ is simple all non-trivial traces on $A$ are densely defined. As in the proof of \cite[Theorem 2.7]{CE}, which builds on \cite{ERS11} and \cite{TT15}, there exists a strictly positive, continuous, linear function $f: T^+(A) \to \mathbb{R}$ which in turn defines an element $\hat{f}$ in the dual cone $L(F(\mathrm{Cu}(A)))$ where  $\hat{f}(d_\tau) := f(\tau)$ for all $\tau \in T^+(A)$.
	
	By hypothesis, there is a $\mathrm{Cu}$-isomorphism $\Phi: \mathrm{Cu}(A) \to \mathrm{Cu}(A \otimes \Z)$. 
	Say $\Phi^*: F(\mathrm{Cu}(A\otimes \Z)) \to F(\mathrm{Cu}(A))$ is the natural map induced by $\Phi$. 
	It follows that $\hat{f} \circ \Phi^* \in L(F(\mathrm{Cu}(A \otimes \Z)))$. 
	By \cite[Theorem 6.6]{ERS11}, there is a element $[x] \in \mathrm{Cu}(A \otimes \Z)$ that induces this function, i.e.\ $\hat f \circ \Phi^* (\lambda) = \lambda ([x])$ for all $\lambda \in F(\mathrm{Cu}(A\otimes \Z))$.  
	Hence, any positive $a \in A \otimes \K$ such that $[a] = \Phi^{-1}([x])$ satisfies 
	\begin{align}
	d_\tau([a]) & = d_\tau \circ \Phi^{-1} ( [x] ) \notag \\
	& = \hat f \circ \Phi^* (d_\tau \circ \Phi^{-1}) \notag \\
	& = \hat f(d_\tau \circ \Phi^{-1} \circ \Phi) = f(\tau),
	\end{align}
	for all $\tau \in T^+(A)$. 
	By \cite[Lemma 2.5]{CE}, the hereditary $\mathrm{C}^*$-subalgebra $B := \overline{a (A\otimes \K)a}$ satisfies that $T^+(B) = T_b(B)$ and $T(B)$ is compact. By \cite[Theorem 2.8]{Br77}, $B$ is stably isomorphic to $A$.	
\end{proof}

With the prerequisite results and technical machinery in place, we now prove Theorems \ref{thm:Main} and \ref{thm:Toms-Winter}. 
\begin{proof}[Proof of Theorem \ref{thm:Main}]
	When $T^+(A) = \{0\}$, the result is known. Indeed, the condition $\mathrm{Cu}(A) \cong \mathrm{Cu}(A \otimes \Z)$ ensures that $\mathrm{Cu}(A)$ is almost unperforated by \cite[Theorem 4.5]{Ro04}.\footnote{When applying results of \cite{Ro04}, note that $\mathrm{Cu}(A) = W(A \otimes \K)$. The notation $W(\cdot)$  refers to a forerunner of the modern Cuntz semigroup $\mathrm{Cu}(\cdot)$ where only matrix amplifications of the algebras are considered instead of the stabilisations.} Hence, $A$ will be purely infinite by \cite[Corollary 5.1]{Ro04}. By \cite[Theorem 3.15]{KP00}, $A \cong A \otimes \mathcal{O}_\infty$. Therefore, $A$ is $\Z$-stable. 
	
Hereinafter, we assume that $T^+(A) \neq \{0\}$. It follows from Brown's Theorem \cite{Br77}, that $A$ is either stably projectionless or stably isomorphic to a unital C$^*$-algebra (see for example \cite[Proposition 2.4]{CE}). Therefore, $A$ is stably isomorphic to a C$^*$-algebra $B$ with $T^+(B) = T_b(B)$ and $T(B)$ non-empty and compact. (If $A$ is stably projectionless use Lemma \ref{lem:nice.her.subalgebra}.)
Since $A$ is simple, separable and nuclear, $B$ is also simple, separable and nuclear. In particular, $B$ will be exact, so Haagerup's Theorem ensures that all lower semicontinuous quasitraces on $B$ are traces. 
	
As the Cuntz semigroup is invariant under stable isomorphism, we have $\mathrm{Cu}(B \otimes \Z) \cong \mathrm{Cu}(B)$. It follows by \cite[Theorem 4.5]{Ro04} that $\mathrm{Cu}(B)$ is almost unperforated and hence $B$ has strict comparison (see for example \cite[Proposition 2.1]{OPR}).
	
Since $A$ has stabilised property $\Gamma$, $B$ has stabilised property $\Gamma$ by Theorem \ref{thm:upGamma.stable}. As $T^+(B) = T_b(B)$ and $T(B)$ is non-empty and compact, $B$ has uniform property $\Gamma$ by Proposition \ref{prop:sGamma.implies.uGamma}. By \cite[Theorem 4.6]{CETW}, $B$ is uniformly McDuff. Hence, $B$ is $\Z$-stable by Theorem \ref{thm:SI}. By \cite[Corollary 3.2]{TW07}, $\Z$-stability is preserved under stable isomorphism. Therefore, $A$ is $\Z$-stable. 
\end{proof}

\begin{proof}[Proof of Theorem \ref{thm:Toms-Winter}]
	The equivalence (i) $\Leftrightarrow$ (ii) is a consequence of \cite{Wi12,Ti14, CETWW, CE}. The direction (iii) $\Rightarrow$ (ii) is the content of Theorem \ref{thm:Main}. 
	
	For (ii) $\Rightarrow$ (iii), let $A$ be a $\Z$-stable $\mathrm{C}^*$-algebra that satisfies the hypotheses of this theorem. Clearly, the condition on the Cuntz semigroup is automatically satisfied. Let us show that $A$ has stabilised property $\Gamma$. This is vacuous in the traceless case because we may take $e_1, \ldots, e_n$ all to be zero. Hence, we may assume that $T^+(A) \neq \{0\}$. By \cite[Theorem 2.7]{CE}, $A$ is stably isomorphic to a hereditary $\mathrm{C}^*$-subalgebra $B \subseteq A\otimes \K$ with $T^+(B) = T_b(B)$ and $T(B)$ compact and non-empty. It follows by \cite[Proposition 2.3]{CETWW} that $B$ has uniform property $\Gamma$ (and hence stabilised property $\Gamma$ by Proposition \ref{prop:sGamma.implies.uGamma}(ii)). Therefore, $A$ has stabilised property $\Gamma$ since this is a stable property.
\end{proof}

We conclude this paper with a final observation. A key component of the proof of Lemma \ref{lem:nice.her.subalgebra} was realising functions $f:T^+(A) \to \mathbb{R}$ with elements in the Cuntz semigroup $\mathrm{Cu}(A)$, i.e.\ finding $[x]\in \mathrm{Cu}(A)$ such that $f(\tau) = d_\tau([x])$. In the proof above, we used the hypothesis that $\mathrm{Cu}(A) \cong \mathrm{Cu}(A \otimes \Z)$ and \cite[Theorem 6.6]{ERS11} to achieve this. Alternatively, we could use stable rank one and \cite[Theorem 7.14]{APRT}. With this modification, we obtain the following variant of Theorem \ref{thm:Main} by an almost identical proof.

\begin{theorem}
	Let $A$ be simple, separable, nuclear $\mathrm{C}^*$-algebra with stable rank one, strict comparison and stabilised property $\Gamma$. Then $A$ is $\Z$-stable.
\end{theorem}

\end{document}